\documentclass[11pt]{article}
\usepackage{amsmath, amssymb, amscd, amsthm, amsfonts}
\usepackage{graphicx}
\usepackage{hyperref}
\usepackage{amssymb}
\usepackage{amsmath}
\usepackage{amsthm}
\usepackage{amsfonts}

\usepackage{a4wide}
\usepackage{longtable}
\usepackage{multirow}

\usepackage[latin1]{inputenc}
\usepackage[english]{babel}
\usepackage[T1]{fontenc}
\usepackage{amssymb,amsmath,amsthm,amstext,amsfonts,latexsym}

\usepackage{pgf,tikz}
\usepackage{mathrsfs}
\usetikzlibrary{arrows}

\newtheorem{proposition}{Proposition}[section]
\newtheorem{theorem}{Theorem}[section]

\theoremstyle{definition}
\newtheorem{definition}{Definition}[section]

\theoremstyle{remark}

\usepackage[all]{xy}
\title{On the capitulation problem of some pure metacyclic fields\\ of degree 20 II}

\date{}

\begin{document}

\maketitle
\begin{center}
{\sc Fouad ELMOUHIB (\textcolor{blue}{the corresponding author}) }\\
{\footnotesize Department of Mathematics and Computer Sciences,\\
Mohammed First University, Oujda, Morocco,\\
\textcolor{blue}{Correspondence: fouad.cd@gmail.com}\\
\vspace{0.7cm}}
{\sc Mohamed TALBI }\\
{\footnotesize Regional Center of Professions of Education and Training,\\
ksirat1971@gmail.com}\\

\vspace{0.7cm}
{\sc Abdelmalek AZIZI }\\
{\footnotesize Department of Mathematics and Computer Sciences,\\
Mohammed First University, Oujda, Morocco,\\
abdelmalekazizi@yahoo.fr}
\end{center}

\begin{abstract}
Let $n$ be a $5^{th}$ power-free naturel number and $k_0\,=\,\mathbb{Q}(\zeta_5)$ be the cyclotomic field generated by a primitive $5^{th}$ root of unity $\zeta_5$. Then $k\,=\,\mathbb{Q}(\sqrt[5]{n},\zeta_5)$ is a pure metacyclic field of absolute degree $20$. In the case that $k$ possesses a $5$-class group $C_{k,5}$ of type $(5,5)$ and all the classes are ambiguous under the action of $Gal(k/k_0)$, the capitulation of $5$-ideal classes of $k$ in its unramified cyclic quintic  extensions is determined.
\end{abstract}

\begin{center}
\section{Introduction}\label{introduction}
\end{center}
Let $k$ be a number field, and $L$ be an unramified abelian extension of $k$. We say that an ideal $\mathcal{I}$ of $k$ or its class capitulates in $L$ if $\mathcal{I}$ becomes principal in $L$.\\
Let $\Gamma\,=\,\mathbb{Q}(\sqrt[5]{n})$ be a pure quintic field, where $n$ is a $5^{th}$ power free naturel number and $k_0\,=\,\mathbb{Q}(\zeta_5)$ be the cyclotomic field generated by a primitive $5^{th}$ root of unity $\zeta_5$. Then $k\,=\,\Gamma(\zeta_5)$ is the normal closure of $\Gamma$. Let $k_5^{(1)}$ be the Hilbert $5$-class field  of $k$, $C_{k,5}$ be the $5$-ideal classes group of $k$, and $C_{k,5}^{(\sigma)}$ be the group of ambiguous ideal classes under the action of $Gal(k/k_0)\,=\,\langle\sigma\rangle$. In the case that $C_{k,5}$ is of type $(5,5)$ and rank $C_{k,5}^{(\sigma)}\,=\,1$, the capitulation of the $5$-ideal classes of $k$ in the six intermediate extensions of $k_5^{(1)}/k$ is determined in [\ref{FOU3}].\\
In this paper, we investigate the capitulation of $5$-ideal classes of $k$ in the unramified cyclic quintic extensions of $k_5^{(1)}/k$, whenever $C_{k,5}$ is of type $(5,5)$ and rank $C_{k,5}^{(\sigma)}\,=\,2$, which mean that all classes are ambiguous. Let $p$ and $q$ primes such that $p\,\equiv\,1\,(\mathrm{mod}\,5)$, $q\,\equiv\,\pm2\,(\mathrm{mod}\,5)$. According to [\ref{FOU1}, theorem 1.1], if $C_{k,5}$ is of type $(5,5)$ and rank $C_{k,5}^{(\sigma)}\,=\,2$, we have three forms of the radicand $n$ as follows:
\begin{itemize}
\item[-] $n\,=\,5^ep\,\not\equiv\,\pm1\pm7\,(\mathrm{mod}\,25)$ with $e\in\{1,2,3,4\}$ and $p\,\not\equiv\,1\,(\mathrm{mod}\,25)$.
\item[-] $n\,=\,p^eq\,\equiv\,\pm1\pm7\,(\mathrm{mod}\,25)$ with $e\in\{1,2,3,4\}$, $p\,\not\equiv\,1\,(\mathrm{mod}\,25)$ and $q\,\not\equiv\,\pm7\,(\mathrm{mod}\,25)$.
\item[-] $n\,=\,p^e\,\equiv\,\pm1\pm7\,(\mathrm{mod}\,25)$ with $e\in\{1,2,3,4\}$ and $p\,\equiv\,1\,(\mathrm{mod}\,25)$.
\end{itemize}
We will study the capitulation of $C_{k,5}$ in the six  intermediate extensions of $k_5^{(1)}/k$ in these cases.
The theoretical results are underpinned by numerical examples obtained with the computational number theory system PARI/GP [\ref{PARI}].
\begin{center}

$\textbf{Notations. \ }$
\end{center}
Throughout this paper, we use the following notations:
\begin{itemize}
 \item The lower case letter $p$ and $q$ denote a prime numbers such that, $p\,\equiv\,1\,(\mathrm{mod}\,5)$ and $q\,\equiv\, \pm2\,(\mathrm{mod}\,5)$.

 \item $\Gamma\,=\,\mathbb{Q}(\sqrt[5]{n})$: a pure quintic field, where $n\neq 1$ is a $5^{th}$ power-free naturel number.
 
 \item $k_0\,=\,\mathbb{Q}(\zeta_5)$: the cyclotomic field, where $\zeta_5\,=\,e^{2i\pi/5}$ is a primitive $5^{th}$ root of unity.
 
 \item $k\,=\,\mathbb{Q}(\sqrt[5]{n},\zeta_5)$: the normal closure of $\Gamma$, a quintic Kummer extension of $k_0$.
 
 
 \item $\langle\tau\rangle\,=\,\operatorname{Gal}(k/\Gamma)$ such that $\tau$ is identity on $\Gamma$, and sends $\zeta_5$ to its square. Hence $\tau$ has order 4.
 
 \item $\langle\sigma\rangle\,=\,\operatorname{Gal}(k/k_0)$ such that $\sigma$ is identity on $k_0$, and sends $\sqrt[5]{n}$ to $\zeta_5\sqrt[5]{n}$. Hence $\sigma$ has order 5.
 
 \item For a number field $L$, denote by:
\begin{itemize}
   \item $\mathcal{O}_{L}$: the ring of integers of $L$;
   \item $C_{L}$, $h_{L}$, $C_{L,5}$: the class group, class number, and $5$-class group of $L$.
   \item $L_5^{(1)}, L^*$: the Hilbert $5$-class field of $L$, and the absolute genus field of $L$. 
   \item $[\mathcal{I}]$: the class of a fractional ideal $\mathcal{I}$ in the class group of $L$.
  \end{itemize}
\item $(\frac{a}{b})_5\,=\,1$  $\Leftrightarrow$ $X^5 \equiv a\,(\mathrm{mod}\, b)$ resolved on $\mathcal{O}_{k_0}$, where $a, b$ are primes in $\mathcal{O}_{k_0}$.
\end{itemize}

\begin{center}
\begin{tikzpicture}
\begin{scope}[xscale=2,yscale=2]
  \node (P) at (0,3) {$k_5^{(1)}$};
  \node (J) at (0.5,1.5) {$K_4$};
  \node (K) at (1,1.5) {$K_5$};
  \node (L) at (1.5,1.5) {$K_6$};
  \node (M) at (-0.5,1.5) {$K_3$};
  \node (N) at (-1,1.5) {$K_2$};
  \node (O) at (-1.5,1.5) {$K_1$};
  
  \node (A) at (0,0) {$k$};
  \draw [-,>=latex] (A) -- (O)   node[midway,below left] {};
  \draw [-,>=latex] (A) -- (N)   node[midway,below left] {};
  \draw [-,>=latex] (A) -- (M)   node[midway,below left] {};
  \draw [-,>=latex] (A) -- (J)   node[midway,below left] {};
  \draw [-,>=latex] (A) -- (K)   node[midway,below left] {};
  \draw [-,>=latex] (A) -- (L)   node[midway,below left] {};  
 
  \draw [-,>=latex] (M) -- (P)   node[midway,below left] {};
  \draw [-,>=latex] (N) -- (P)   node[midway,below left] {};
  \draw [-,>=latex] (O) -- (P)   node[midway,below left] {};
  \draw [-,>=latex] (J) -- (P)   node[midway,below left] {};
  \draw [-,>=latex] (K) -- (P)   node[midway,below left] {};
  \draw [-,>=latex] (L) -- (P)   node[midway,below left] {};  
  
\end{scope}
\end{tikzpicture}
\end{center}
\begin{center}
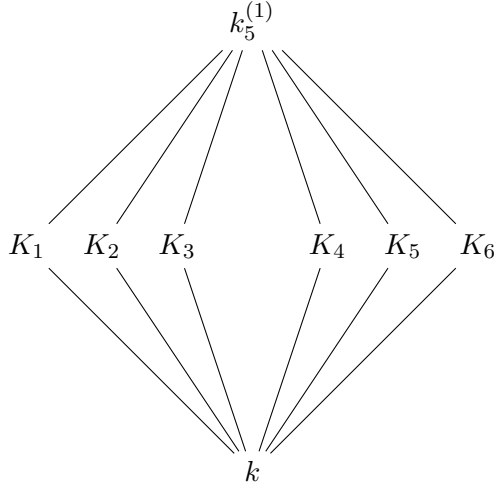
Figure 1: The unramified quintic sub-extensions of $k_5^{(1)}
/k$
\end{center}

\begin{center}
\section{Preliminaries}
\end{center}
\subsection{Decomposition laws in Kummer extension}
Since the extensions of $k$ and $k_0$ are all Kummer's extensions, we recall the decomposition laws of ideals in these extensions.
\begin{proposition}.\label{prKummer}Let $L$ a number field contains the $l^{th}$ root of unity, where $l$ is prime, and $\theta$ element of $L$, such that $\theta\,\neq\,\mu^l$, for all $\mu \in L$, therefore $L(\sqrt[l]{\theta})$ is cyclic extension of degree $l$ over $L$. We note by $\zeta$ a $l^{th}$ primitive root of unity.\\ 
(1) We assume that a prime $\mathcal{P}$ of $L$, divides exactly $\theta$ to the power $\mathcal{P}^a$.
\begin{itemize}
\item If $a=0$ and $\mathcal{P}$ don't divides $l$, then $\mathcal{P}$ split completly in $L(\sqrt[l]{\theta})$ when the congruence\\ $\theta\,\equiv\,X^l\,(\mathrm{mod}\,\mathcal{P})$ has solution in $L$.
\item If $a=0$ and $\mathcal{P}$ don't divides $l$, then $\mathcal{P}$ is inert in $L(\sqrt[l]{\theta})$ when the congruence $\theta\,\equiv\,X^l\,(\mathrm{mod}\,\mathcal{P})$ has no solution in $L$.
\item If $l\nmid a$, then $\mathcal{P}$ is totaly ramified in $L(\sqrt[l]{\theta})$.
\end{itemize}
(2) Let $\mathcal{B}$ a prime factor of $1-\zeta$ that divides $1-\zeta$ exactly to the a$^{th}$ power. Suppose that $\mathcal{B}\nmid\theta$, then $\mathcal{B}$ split completly in $L(\sqrt[l]{\theta})$ if the congruence
\begin{center}
$\theta\,\equiv\,X^l\,(\mathrm{mod}\,\mathcal{B}^{al+1})$ \hspace{2cm} $(*)$ 
\end{center} 
has solution in $L$. the ideal $\mathcal{B}$ is inert in $L(\sqrt[l]{\theta})$ if the congruence
\begin{center}
$\theta\,\equiv\,X^l\,(\mathrm{mod}\,\mathcal{B}^{al})$ 
\hspace{2cm} $(**)$
\end{center}
has solution in $L$, without $(*)$ has. The ideal $\mathcal{B}$ is totaly ramified in $L$ if the congruence $(**)$ has no solution.
\end{proposition}
\begin{proof}
see [\ref{Hec}].
\end{proof}

\subsection{Relative genus field $(k/k_0)^*$ of $k$ over $k_0$}
Let $\Gamma\,=\,\mathbb{Q}(\sqrt[5]{n})$ be a pure quintic field, $k_0\,=\,\mathbb{Q}(\zeta_5)$ the $5^{th}$-cyclotomic field and $k\,=\,\Gamma(\zeta_5)$ be the normal closure of $\Gamma$. The relative genus field $(k/k_0)^*$ of $k$ over $k_0$ is the maximal abelian extension of $k_0$ which is contained in the Hilbert $5$-class field $k_5^{(1)}$ of $k$. Let $q^* \in \{0, 1, 2\}$  such that\\ $q^*$ = $ \begin{cases}
2 & \text{if }\, \zeta_5 , \zeta_5+1\, \text{are norm of element in}\, k-\{0\}.\\
1 & \text{if}\,\, \zeta_5^i(\zeta_5+1)^j\, \text{is the norm of an element in}\,k-\{0\}\, \text{for some exponents i and j}.\\
0 & \text{if for no exponents i, j  the element}\,\, \zeta_5^i(\zeta_5+1)^j\, \text{is a norm from}\,\, k-\{0\}.\\
\end{cases}$
\begin{proposition}.\label{prop genre}
 Let $k\,=\,k_0(\sqrt[5]{n})$ such that $n$ = $\mu\lambda^{e_{\lambda}}\pi_{1}^{e_1}....\pi_{f}^{e_f}\pi_{f+1}^{e_{f+1}}.....\pi_{g}^{e_g}$ in $k_0$, where $\mu$ is unity
of $\mathcal{O}_{k_0}$, $\lambda\,=\,1-\zeta_5$ the unique prime above 5 in $k_0$ and each prime $\pi_i \,\equiv\, \pm1,\pm7\, (\mathrm{mod}\,\lambda^5)$ for $1\leq i \leq f$ and $\pi_j \,\not\equiv\, \pm1,\pm7 \, (\mathrm{mod}\,\lambda^5)$ for $f+1\leq j \leq g$. Then we have:
\begin{itemize}
\item[$(i)$] there exists $h_i \in \{1,..,4\}$ such that $\pi_{f+1}\pi_i^{h_i}\,\equiv\, \pm1,\pm7\, (\mathrm{mod}\,\lambda^5)$, for $f+2\leq i \leq g$.
\item[$(ii)$] if $n\,\not\equiv\,\pm1\pm7\, (\mathrm{mod}\,\lambda^5)$ and $q^*\,=\,1$, then the genus field $(k/k_0)^*$ is given as:
\begin{center}
$(k/k_0)^*$ = $k(\sqrt[5]{\pi_1},....\sqrt[5]{\pi_f},\sqrt[5]{\pi_{f+1}\pi_{f+2}^{h_{f+2}}},....\sqrt[5]{\pi_{f+1}\pi_{g}^{h_{g}}})$
\end{center}
where $h_i$ is chosen as in  (i).
\item[$(iii)$] in the other cases of $q^*$ and the congruence of $n$, the genus field  $(k/k_0)^*$ is given by deleting an appropriate number of $5^{th}$ root from the right side of (ii).
\end{itemize}
\end{proposition}
\begin{proof}
see [\ref{Mani}, proposition 5.8].
\end{proof}

\begin{center}
\section{Study of capitulation}
\end{center}
This being the case, let $\Gamma$, $k_0$ and $k$ as above. If $C_{k,5}$ is of type $(5,5)$ and the group of ambigous classes $C_{k,5}^{(\sigma)}$ under the action of $Gal(k/k_0)\,=\,\langle\sigma\rangle$ has rank 2, we have  $C_{k,5}\,=\,C_{k,5}^{(\sigma)}$. By class field theory $C_{k,5}^{1-\sigma}$ correspond to $(k/k_0)^*$, and since $C_{k,5}\,=\,C_{k,5}^{(\sigma)}$ we get that $C_{k,5}^{1-\sigma}\,=\,\{1\}$, hence $(k/k_0)^*\,=\,k_5^{(1)}$ is the Hilbert $5$-class field of $k$.\\
When $C_{k,5}$ is of type $(5,5)$, it has $6$ subgroups of order $5$, denoted $H_i$, $1\leq i \leq 6$. Let $K_i$ be the intermediate extension of $k_5^{(1)}/k$, corresponding by class field theory to $H_i$. Its easy to see that $C_{k,5}\,\cong\, C_{k,5}^+\times C_{k,5}^-$ such that $C_{k,5}^{+}\,=\,\{\mathcal{A}\in C_{k,5}\,|\,\mathcal{A}^{\tau^2} = \mathcal{A}\}$ and $C_{k,5}^{-}\,=\,\{\mathcal{X}\in C_{k,5}\,|\,\mathcal{X}^{\tau^2} = \mathcal{X}^{-1}\}$ with $Gal(k/\Gamma)\,=\,\langle\tau\rangle$. As each $K_i$ is cyclic of order $5$ over $k$, there is at least one subgroup of order $5$ of $C_{k,5}$, i.e at least one $H_l$ for some $l \in \{1,2,3,4,5,6\}$, which capitulates in $K_i$ (by Hilbert's theorem $94$).
\begin{definition}
Let $\mathcal{S}_j$ be a generator of $H_j$ $(1\leq j \leq 6)$ corresponding to $K_j$. For $1\leq j \leq 6$, let $i_j \in  \{0,1,2,3,4,5,6\}$. We say that the capitulation is of type $(i_1,i_2,i_3,i_4,i_5,i_6)$ to mean the following:
\item[$(1)$] when $i_j \in \{1,2,3,4,5,6\}$, then only the class $\mathcal{S}_{i_j}$ and its powers capitulate in $K_j$;
\item[$(2)$] when $i_j\,=\,0$, then all the $5$-classes capitulate in $K_j$.
\end{definition}

Throught the paper we order the subgroups $H_i$ of $C_{k,5}$ as follows:\\
$H_1\,=\,C_{k,5}^+\,=\,\langle\mathcal{A}\rangle$, $H_6\,=\,C_{k,5}^-\,=\,\langle\mathcal{X}\rangle$, $H_2\,=\,\langle\mathcal{AX}\rangle$, $H_3\,=\,\langle\mathcal{AX}^2\rangle$, $H_4\,=\,\langle\mathcal{AX}^3\rangle$ and $H_5\,=\,\langle\mathcal{AX}^4\rangle$. By class field theory we have $H_6$ correspond to $K_6\,=\,k\Gamma_5^{(1)}$, with $\Gamma_5^{(1)}$ is the Hilbert $5$-class field of $\Gamma$. By the action of $Gal(k/\mathbb{Q})$ on $C_{k,5}$, we can give the following:

\begin{proposition}\label{tau permut}For all continuations of the automorphisms $\sigma$ and $\tau$ we have:
\item[$(1)$] $K_i^{\sigma}\,=\,K_i$ $(i\,=\,1,2,3,4,5,6)$ i.e $\sigma$ sets all $K_i$

\item[$(2)$] $K_1^{\tau^2}\,=\,K_1$, $K_6^{\tau^2}\,=\,K_6$, $K_2^{\tau^2}\,=\,K_5$ and $K_3^{\tau^2}\,=\,K_4$. i.e $\tau^2$ sets $K_1$, $K_6$ and permutes $K_2$ with $K_5$ and $K_3$ with $K_4$.
\end{proposition}

\begin{proof}
We will agree that for all $1\leq i \leq 6$, and for all $w\in Gal(k/\mathbb{Q})$ we have $H_i^w\,=\,\{\mathcal{C}^w\,|\, \mathcal{C}\in H_i\}$. 
\item[$(1)$] Since all classes are ambigous because $C_{k,5}\,=\,C_{k,5}^{(\sigma)}$, then $\sigma$ sets all $H_i$.
\item[$(2)$]We have $H_1\,=\,C_{k,5}^+\,=\,\langle\mathcal{A}\rangle$ and $H_6\,=\,C_{k,5}^-\,=\,\langle\mathcal{X}\rangle$, then $H_1^{\tau^2}\,=\,H_1$ and $H_6^{\tau^2}\,=\,H_6$.\\
- Since $(\mathcal{AX})^{\tau^2}\,=\,\mathcal{A}^{\tau^2}\mathcal{X}^{\tau^2}\,=\,\mathcal{AX}^{-1}\,=\,\mathcal{AX}^4 \in H_5$ then $H_2^{\tau^2}\,=\,H_5$.\\
- Since $(\mathcal{A}\mathcal{X}^2)^{\tau^2}\,=\,\mathcal{A}^{\tau^2}(\mathcal{X}^2)^{\tau^2}\,=\,\mathcal{AX}^{-2}\,=\,\mathcal{AX}^3 \in H_4$ then $H_3^{\tau^2}\,=\,H_4$.\\
- Since $\tau^4\,=\,1$ we get that $H_5^{\tau^2}\,=\,H_2$ and $H_4^{\tau^2}\,=\,H_3$.\\
The relations between the fields $K_i$ in $(1)$ and $(2)$ are nothing else than the translations of the corresponding relations for the subgroups $H_i$ via class field theory.
\end{proof}
To study the capitulation problem of $k$ whenever $C_{k,5}$ is of type $(5,5)$ and $C_{k,5}\,=\,C_{k,5}^{(\sigma)}$, we will investigate the three forms of the radicand $n$ proved in [\ref{FOU1}, theorem 1.1] and mentioned above.

\subsection{The case $n\,=\,p^e\,\equiv\,\pm1\pm7\,(\mathrm{mod}\,25)$, where $p\,\equiv\,1\,(\mathrm{mod}\,25)$}
Let $k\,=\,\Gamma(\zeta_5)$ be the normal closure of $\Gamma\,=\,\mathbb{Q}(\sqrt[5]{n})$, where $n\,=\,p^e$ such that $p\,\equiv\,1\,(\mathrm{mod}\,25)$ and $e\in \{1,2,3,4\}$. Since $p\,\equiv\,1\,(\mathrm{mod}\,5)$ we have that $p$ splits completely in $k_0\,=\,\mathbb{Q}(\zeta_5)$ as $p\,=\,\pi_1\pi_2\pi_3\pi_4$, with $\pi_i$ are primes in $k_0$ such that $\pi_i\,\equiv\,1\,(\mathrm{mod}\,5\mathcal{O}_{k_0})$, then the primes of $k_0$ ramified in $k$ are $\pi_i$.\\
If $\mathcal{P}_1, \mathcal{P}_2, \mathcal{P}_3$ and $\mathcal{P}_4$ are respectivly the prime ideals of $k$ above $\pi_1, \pi_2, \pi_3$ and $\pi_4$, then $\mathcal{P}_i^5\,=\,\pi_i\mathcal{O}_k\, (i=1,2,3,4)$  and since $\tau$ acte transitively on $\pi_i$, we have that $\tau^2$ permutes $\pi_1$ with $\pi_3$, hence $\tau^2$ permutes $\mathcal{P}_1$ with $\mathcal{P}_3$. Since $\pi_i^{\sigma}\,=\,\pi_i$, we have $\mathcal{P}_i^{\sigma}\,=\,\mathcal{P}_i$. In fact $[\mathcal{P}_i]\, (i=1,2,3,4)$ generate the group of strong ambigous ideal classes denoted $C_{k,s}^{(\sigma)}$. The next theorem allow us to determine explicitly the intermediate extensions of $k_5^{(1)}/k$.
\begin{theorem}\label{6 exten case 1}
Let $k$ and $n$ as above. Let $\pi_1, \pi_2, \pi_3$ and $\pi_4$ a primes of $k_0$ congrus to $1$ modulo $\lambda^5$ such that $p\,=\,\pi_1\pi_2\pi_3\pi_4$, then:
\item[$(1)$] $k_5^{(1)}\,=\, k(\sqrt[5]{\pi_1}, \sqrt[5]{\pi_3})$.
\item[$(2)$] The six intermediate extensions of $k_5^{(1)}/k$ are:  $k(\sqrt[5]{\pi_1})$, $k(\sqrt[5]{\pi_3})$, $k(\sqrt[5]{\pi_1\pi_3})$, $k(\sqrt[5]{\pi_1\pi_3^2})$, $k(\sqrt[5]{\pi_1\pi_3^3})$ and $k(\sqrt[5]{\pi_1\pi_3^4})$. Furthermore $\tau^2$ permutes $k(\sqrt[5]{\pi_1})$ with $k(\sqrt[5]{\pi_3})$ and $k(\sqrt[5]{\pi_1\pi_3^2})$ with $k(\sqrt[5]{\pi_1\pi_3^3})$, and sets $k(\sqrt[5]{\pi_1\pi_3})$, $k(\sqrt[5]{\pi_1\pi_3^4})$.
\end{theorem}

\begin{proof}
\item[$(1)$] We have that $k_5^{(1)}\,=\,(k/k_0)^*$. Since $k\,=\,k_0(\sqrt[5]{n})$ with $n\,=\,p\,=\pi_1\pi_2\pi_3\pi_4$ in $k_0$ and $\pi_i\,\equiv\,1\,(\mathrm{mod}\,\lambda^5)\, (i\,=\,1,2,3,4)$, then by proposition \ref{prop genre} we have $(k/k_0)^*\,=\,k(\sqrt[5]{\pi_1}, \sqrt[5]{\pi_3})$.

\item[$(2)$] If $k_5^{(1)}\,=\,k(\sqrt[5]{\pi_1}, \sqrt[5]{\pi_3})$, then the six intermediate extensions are: $k(\sqrt[5]{\pi_1})$, $k(\sqrt[5]{\pi_3})$, $k(\sqrt[5]{\pi_1\pi_3})$, $k(\sqrt[5]{\pi_1\pi_3^2})$, $k(\sqrt[5]{\pi_1\pi_3^3})$ and $k(\sqrt[5]{\pi_1\pi_3^4})$. We have $\tau^2(\pi_1)\,=\,\pi_3$ then its easy to see that $\tau^2$ sets the fields $k(\sqrt[5]{\pi_1\pi_3})$, $k(\sqrt[5]{\pi_1\pi_3^4})$. Since $\tau^2(\pi_1)\,=\,\tau^2(\sqrt[5]{\pi_1^5})\,=\,(\tau^2(\sqrt[5]{\pi_1}))^5\,=\,\pi_3$, then $\tau^2(\sqrt[5]{\pi_1})$ is $5^{th}$ root of $\pi_3$. Hence $k(\sqrt[5]{\pi_3})\,=\,k(\tau^2(\sqrt[5]{\pi_1}))$ i.e $k(\sqrt[5]{\pi_3})\,=\,k(\sqrt[5]{\pi_1})^{\tau^2}$. By the same resoning we prove that $k(\sqrt[5]{\pi_1})\,=\,k(\sqrt[5]{\pi_3})^{\tau^2}$. Hence $\tau^2$ permutes $k(\sqrt[5]{\pi_1})$ with $k(\sqrt[5]{\pi_3})$.\\
We have $\tau^2(\pi_1\pi_3^2)\,=\,\pi_1^2\pi_3$ then  $\tau^2(\pi_1\pi_3^2)\,=\,\tau^2(\sqrt[5]{(\pi_1\pi_3^2)^5})\,=\,(\tau^2(\sqrt[5]{\pi_1\pi_3^2}))^5\,=\,\pi_1^2\pi_3$, hence $\tau^2(\sqrt[5]{\pi_1\pi_3^2})$ is $5^{th}$ root of $\pi_1^2\pi_3$. Then $k(\sqrt[5]{\pi_1^2\pi_3})\,=\,k(\tau^2(\sqrt[5]{\pi_1\pi_3^2}))$ i.e $k(\sqrt[5]{\pi_1^2\pi_3})\,=\,k(\sqrt[5]{\pi_1\pi_3^3})\,=\,k(\sqrt[5]{\pi_1\pi_3^2})^{\tau^2}$. By the same resoning we prove that $k(\sqrt[5]{\pi_1\pi_3^2})\,=\,k(\sqrt[5]{\pi_1\pi_3^3})^{\tau^2}$. Hence $\tau^2$ permutes $k(\sqrt[5]{\pi_1\pi_3^2})$ with $k(\sqrt[5]{\pi_1\pi_3^3})$.
\end{proof}

The generators of $C_{k,5}$ when its of type $(5,5)$ and the radicand $n$ is as above are determined as follows:

\begin{theorem}\label{generator case 1} Let $k$ and $n$ as above. Let $\pi_1, \pi_2, \pi_3$ and $\pi_4$ a primes of $k_0$ congrus to 1 $(\mathrm{mod}\, \lambda^5)$ such that $n\,=\,p\,=\,\pi_1\pi_2\pi_3\pi_4$. Let $\mathcal{P}_1, \mathcal{P}_2, \mathcal{P}_3$ and $\mathcal{P}_4$ prime ideals of $k$ such that $\mathcal{P}_i^5\,=\,\pi_i\mathcal{O}_{k_0}\, (i=1,2,3,4)$. Then:
\begin{center}
$C_{k,5}\,=\,\langle [\mathcal{P}_1\mathcal{P}_3], [\mathcal{P}_1\mathcal{P}_3^4]\rangle$
\end{center}
\end{theorem}

\begin{proof}
According to [\ref{FOU1}, theorem 1.1], for that case of the radicand $n$, we have that $\zeta_5^i(1+\zeta_5)^j$ is norm of element in $k-\{0\}$. By [\ref{Mani}, section 5.3], if $\zeta_5$ is not norm of unit of $k$ we have $C_{k,5}\,=\,C_{k,5}^{(\sigma)}\,\neq\,C_{k,s}^{(\sigma)}$, so $C_{k,s}^{(\sigma)}$ contained in $C_{k,5}^{(\sigma)}$. Hence we discuss two cases:
\item[-] $1^{th}$ case: $C_{k,5}\,=\,C_{k,5}^{(\sigma)}\,\neq\,C_{k,s}^{(\sigma)}$: In this case, $C_{k,s}^{(\sigma)}$ is contained in $C_{k,5}\,=\,C_{k,5}^{(\sigma)}$, and by [\ref{Mani}, section 5.3] we have $C_{k,5}^{(\sigma)}/C_{k,s}^{(\sigma)}\,=\,C_{k,5}/C_{k,s}^{(\sigma)}$ is cyclic group of order $5$. Since $C_{k,5}$ has order $25$ then $C_{k,s}^{(\sigma)}$ is cyclic of order $5$. We have that $C_{k,s}^{(\sigma)}\,=\,\langle [\mathcal{P}_1], [\mathcal{P}_2], [\mathcal{P}_3], [\mathcal{P}_4] \rangle$, $\mathcal{P}_1^{\tau^2}\,=\,\mathcal{P}_3$ and $\mathcal{P}_2^{\tau^2}\,=\,\mathcal{P}_4$, so $\mathcal{P}_1$ and $\mathcal{P}_2$ can not be both principals in $k$, otherwise $\mathcal{P}_3\,=\,\mathcal{P}_1^{\tau^2}$ and $\mathcal{P}_4\,=\,\mathcal{P}_2^{\tau^2}$ will be principals too, hence $C_{k,s}^{(\sigma)}\,=\,\{1\}$, which is impossible. by the same reasoning we have that $\mathcal{P}_3$ and $\mathcal{P}_4$ can not be both principals in $k$. Since $C_{k,s}^{(\sigma)}$ is cyclic of order $5$  and without loosing generality we get that $C_{k,s}^{(\sigma)}\,=\,\langle [\mathcal{P}_1]\rangle$, so $\mathcal{P}_1$ and $\mathcal{P}_3\,=\,\mathcal{P}_1^{\tau^2}$ are not principals. Since $C_{k,5} \cong C_{k,5}^+\times C_{k,5}^-$ its sufficient to find generators of $C_{k,5}^+$ and $C_{k,5}^-$. As $[\mathcal{P}_1\mathcal{P}_3]^{\tau^2}\,=\,[(\mathcal{P}_1\mathcal{P}_3)^{\tau^2}]\,=\,[\mathcal{P}_1\mathcal{P}_3]$ then $C_{k,5}^+\,=\,\langle[\mathcal{P}_1\mathcal{P}_3]\rangle$, and $[\mathcal{P}_1\mathcal{P}_3^4]^{\tau^2}\,=\,[(\mathcal{P}_1\mathcal{P}_3^4)^{\tau^2}]\,=\,[\mathcal{P}_1^4\mathcal{P}_3]\,=\,[\mathcal{P}_1\mathcal{P}_3^4]^{-1}$ then $C_{k,5}^-\,=\,\langle[\mathcal{P}_1\mathcal{P}_3^4]\rangle$. Hence $C_{k,5}\,=\,\langle [\mathcal{P}_1\mathcal{P}_3], [\mathcal{P}_1\mathcal{P}_3^4]\rangle$.

\item[-] $2^{th}$ case: $C_{k,5}\,=\,C_{k,5}^{(\sigma)}\,=\,C_{k,s}^{(\sigma)}$: We admit the same reasoning of $1^{th}$ case because none of $\mathcal{P}_i\, (i=1,2,3,4)$ is principal, otherwise $C_{k,5}\,=\,C_{k,s}^{(\sigma)}\,=\,\{1\}$, which is impossible. Hence $C_{k,5}\,=\,\langle [\mathcal{P}_1\mathcal{P}_3], [\mathcal{P}_1\mathcal{P}_3^4]\rangle$.
\end{proof}
Now we are able to stat the main theorem of capitulation in this case.
\begin{theorem}\label{theo capi case 1}
We keep the same assumptions as theorem \ref{generator case 1} Then:
\item[$(1)$] If $(\frac{\pi_1}{\pi_3})_5\,=\,1$ then $K_1\,=\,k(\sqrt[5]{\pi_1\pi_3})$ or $k(\sqrt[5]{\pi_1\pi_3^4})$, $K_2\,=\,k(\sqrt[5]{\pi_3})$, $K_3\,=\,k(\sqrt[5]{\pi_1\pi_3^2})$ or $k(\sqrt[5]{\pi_1\pi_3^3})$, $K_4\,=\,k(\sqrt[5]{\pi_1\pi_3^3})$ or $k(\sqrt[5]{\pi_1\pi_3^2})$, $K_5\,=\,k(\sqrt[5]{\pi_1})$ and $K_6\,=\,k(\sqrt[5]{\pi_1\pi_3^4})$ or $k(\sqrt[5]{\pi_1\pi_3})$. Otherwise we just permute $K_2$ and $K_5$.

\item[$(2)$] $[\mathcal{P}_1\mathcal{P}_3]$ capitulates in $k(\sqrt[5]{\pi_1\pi_3})$, $[\mathcal{P}_i]$ capitulates in $k(\sqrt[5]{\pi_i})\, (i= 1, 3)$, $[\mathcal{P}_1\mathcal{P}_3^2]$ capitulates in $k(\sqrt[5]{\pi_1\pi_3^2})$, $[\mathcal{P}_1\mathcal{P}_3^3]$ capitulates in $k(\sqrt[5]{\pi_1\pi_3^3})$ and $[\mathcal{P}_1\mathcal{P}_3^4]$ capitulates in $k(\sqrt[5]{\pi_1\pi_3^4})$.

\item[$(3)$] $(i)$  If $(\frac{\pi_1}{\pi_3})_5\,=\,1$ and $K_6\,=\,k(\sqrt[5]{\pi_1\pi_3^4})$ then the possible types of capitulation are: $(0,0,0,0,0,0)$, $(1,0,0,0,0,0)$, 
$(0,2,0,0,5,0)$, 
$(1,2,0,0,5,0)$,  
$\{(0,0,3,4,0,0)$ or $(0,0,4,3,0,0)\}$, 
$\{(1,0,3,4,0,0)$ or $(1,0,4,3,0,0)\}$, 
$\{(0,2,3,4,5,0)$ or $(0,2,4,3,5,0)\}$, 
$\{(1,2,3,4,5,0)$ or $(1,2,4,3,5,0)\}$.\\

$(ii)$ If $(\frac{\pi_1}{\pi_3})_5\,=\,1$ and $K_6\,=\,k(\sqrt[5]{\pi_1\pi_3})$ then the same possible types of capitulation accur as in $(i)$ with $i_6 = 0$ or $1$ and $i_1 = 0$ or $6$\\

- $(iii)$ If $(\frac{\pi_1}{\pi_3})_5\,\neq\,1$ then the same possible types of capitulation accur as $(i)$ and $(ii)$ by permuting 2 and 5.
\end{theorem}

\begin{proof}
\item[$(1)$] According to theorem \ref{6 exten case 1}, we have that $\tau^2$ permutes $k(\sqrt[5]{\pi_1})$ with $k(\sqrt[5]{\pi_3})$ and $k(\sqrt[5]{\pi_1\pi_3^2})$ with $k(\sqrt[5]{\pi_1\pi_3^3})$, and sets $k(\sqrt[5]{\pi_1\pi_3})$, $k(\sqrt[5]{\pi_1\pi_3^4})$. By class field theory $K_i$ correspond to $H_i\,(i=1,2,3,4,5,6)$, for that we determine explicitly the six subgroups $H_i$ of $C_{k,5}$ as follows:\\
We have that $C_{k,5}\,=\,\langle\mathcal{A}, \mathcal{X}\rangle$, where $H_1\,=\,C_{k,5}^+\,=\,\langle \mathcal{A} \rangle$ and $H_6\,=\,C_{k,5}^-\,=\,\langle \mathcal{X} \rangle$. By theorem \ref{generator case 1} we have $\mathcal{A}\,=\,[\mathcal{P}_1\mathcal{P}_3]$ and $\mathcal{X}\,=\,[\mathcal{P}_1\mathcal{P}_3^4]$, then $\mathcal{A}\mathcal{X}\,=\,[\mathcal{P}_1]^2$, $\mathcal{A}\mathcal{X}^2\,=\,[\mathcal{P}_1\mathcal{P}_3^3]^3$, $\mathcal{A}\mathcal{X}^3\,=\,[\mathcal{P}_1\mathcal{P}_3^2]^4$ and $\mathcal{A}\mathcal{X}^4\,=\,[\mathcal{P}_3]^4$. Hence $H_2\,=\,\langle [\mathcal{P}_1] \rangle$, $H_3\,=\,\langle [\mathcal{P}_1\mathcal{P}_3^3] \rangle$, $H_4\,=\,\langle [\mathcal{P}_1\mathcal{P}_3^2] \rangle$ and $H_5\,=\,\langle [\mathcal{P}_3] \rangle$. Since $\tau^2$ sets $k(\sqrt[5]{\pi_1\pi_3})$ and $k(\sqrt[5]{\pi_1\pi_3^4})$, if $K_1\,=\,k(\sqrt[5]{\pi_1\pi_3})$ then  $K_6\,=\,k(\sqrt[5]{\pi_1\pi_3^4})$ and vise versa. If $(\frac{\pi_1}{\pi_3})_5\,=\,1$ then $X^5 \equiv \pi_1\,(\mathrm{mod}\, \pi_3)$ resolved on $\mathcal{O}_{k_0}$ and by proposition \ref{prKummer} we have that $\pi_1$ splits completly in $k_0(\sqrt[5]{\pi_3})$, which equivalent to say that $\mathcal{P}_1$ splits completly in $k(\sqrt[5]{\pi_3})$, hence $K_2\,=\,k(\sqrt[5]{\pi_3})$ and we get that $K_5\,=\,k(\sqrt[5]{\pi_1})$ and if $K_3\,=\,k(\sqrt[5]{\pi_1\pi_3^2})$ then  $K_4\,=\,k(\sqrt[5]{\pi_1\pi_3^3})$ and vise versa. Since $\pi_1$ and $\pi_3$ divide $\pi_1\pi_3$, $\pi_1\pi_3^2$, $\pi_1\pi_3^3$ and $\pi_1\pi_3^4$, if $(\frac{\pi_1}{\pi_3})_5\,\neq\,1$ then $K_2\,=\,k(\sqrt[5]{\pi_1})$ and $K_5\,=\,k(\sqrt[5]{\pi_3})$.

\item[$(2)$]- Since $\mathcal{P}_i^5\,=\,\pi_i\mathcal{O}_k$ $i=1,3$ we have $(\mathcal{P}_1\mathcal{P}_3)^5\,=\,\pi_1\pi_3\mathcal{O}_k$, then $(\mathcal{P}_1\mathcal{P}_3)^5\,=\,\pi_1\pi_3\mathcal{O}_{k(\sqrt[5]{\pi_1\pi_3})}$ in $k(\sqrt[5]{\pi_1\pi_3})$ and $\pi_1\pi_3\mathcal{O}_{k(\sqrt[5]{\pi_1\pi_3})}\,=\,(\sqrt[5]{\pi_1\pi_3}\mathcal{O}_{k(\sqrt[5]{\pi_1\pi_3})})^5$, hence $\mathcal{P}_1\mathcal{P}_3\mathcal{O}_{k(\sqrt[5]{\pi_1\pi_3})}\,=\,\sqrt[5]{\pi_1\pi_3}\mathcal{O}_{k(\sqrt[5]{\pi_1\pi_3})}$. Thus $\mathcal{P}_1\mathcal{P}_3$ seen in $\mathcal{O}_{k(\sqrt[5]{\pi_1\pi_3})}$ becomes principal, i.e $[\mathcal{P}_1\mathcal{P}_3]$ capitulates in $k(\sqrt[5]{\pi_1\pi_3})$.\\
- Since $(\mathcal{P}_1\mathcal{P}_3^2)^5\,=\,\pi_1\pi_3^2\mathcal{O}_k$, we have $(\mathcal{P}_1\mathcal{P}_3^2)^5\,=\,\pi_1\pi_3^2\mathcal{O}_{k(\sqrt[5]{\pi_1\pi_3^2})}$ in $k(\sqrt[5]{\pi_1\pi_3^2})$ and $\pi_1\pi_3^2\mathcal{O}_{k(\sqrt[5]{\pi_1\pi_3^2})}\,=\,(\sqrt[5]{\pi_1\pi_3^2}\mathcal{O}_{k(\sqrt[5]{\pi_1\pi_3^2})})^5$, hence $\mathcal{P}_1\mathcal{P}_3^2\mathcal{O}_{k(\sqrt[5]{\pi_1\pi_3^2})}\,=\,\sqrt[5]{\pi_1\pi_3^2}\mathcal{O}_{k(\sqrt[5]{\pi_1\pi_3^2})}$. Thus $\mathcal{P}_1\mathcal{P}_3^2$ seen in $\mathcal{O}_{k(\sqrt[5]{\pi_1\pi_3^2})}$ becomes principal, i.e $[\mathcal{P}_1\mathcal{P}_3^2]$ capitulates in $k(\sqrt[5]{\pi_1\pi_3^2})$. By the same reasoning we have $[\mathcal{P}_1\mathcal{P}_3^3]$ capitulates in $k(\sqrt[5]{\pi_1\pi_3^3})$ and  $[\mathcal{P}_1\mathcal{P}_3^4]$ capitulates in $k(\sqrt[5]{\pi_1\pi_3^4})$ .\\
- We have $\mathcal{P}_1^5\,=\,\pi_1\mathcal{O}_{k}$, then $\mathcal{P}_1\mathcal{O}_{k(\sqrt[5]{\pi_1})}\,=\,\sqrt[5]{\pi_1}\mathcal{O}_{k(\sqrt[5]{\pi_1})}$. Hence $[\mathcal{P}_1]$ capitulates in $k(\sqrt[5]{\pi_1})$. By the same reasoning we have $[\mathcal{P}_3]$ capitulates in $k(\sqrt[5]{\pi_3})$.\\

\item[$(3)$] $(i)$ If $(\frac{\pi_1}{\pi_3})_5\,=\,1$ and $K_6\,=\,k(\sqrt[5]{\pi_1\pi_3^4})$ we have $[\mathcal{P}_1\mathcal{P}_3^4]$ capitulates in $K_6$. According to [\ref{Mani}, Lemma 6.2] we have that $C_{k,5}^+ \cong C_{\Gamma,5}$. We denote by $j_{k/\Gamma}: C_{\Gamma,5} \longleftarrow C_{k,5}$ the homomorphism induced by extension of ideals of $\Gamma$ in $k$. By class field theory we have that $C_{k,5}^-$ correspond to $K_6\,=\,k\Gamma_5^{(1)}$, and since $C_{k,5}^+\,=\,\langle[\mathcal{P}_1\mathcal{P}_3]\rangle$ and $\mathcal{P}_1\mathcal{P}_3\,=\,j_{k/\Gamma}(\mathcal{J})$ such that $C_{\Gamma,5}\,=\,\langle\mathcal{J}\rangle$, then $[\mathcal{P}_1\mathcal{P}_3]$ capitulates in $K_6$, and as $C_{k,5}\,=\,\langle [\mathcal{P}_1\mathcal{P}_3], [\mathcal{P}_1\mathcal{P}_3^4]\rangle$, then all classes capitulate in $K_6\,=\,k(\sqrt[5]{\pi_1\pi_3^4})$. We determine possible types of capitulation  $(i_1,i_2,i_3,i_4,i_5,i_6)$. We have that $i_6\,=\,0$, $K_2\,=\,K_5^{\tau^2}$, $K_3\,=\,K_4^{\tau^2}$ and $C_{k,5}\,=\,C_{k,5}^{\tau^2}$. If $i_1\neq 0$ we have $i_1 = 1$, if $i_2\neq 0$ we have $i_2 = 2$ and if $i_5\neq 0$ we have $i_5 = 5$. $i_3$ and $i_4$ are both nulls or non nulls, so if $i_3$ and $i_4$ $\neq 0$, then $(i_3, i_4)\,=\,(3,4)$ or $(4,3)$. Thus the possible types of capitulation are $(0,0,0,0,0,0)$, $(1,0,0,0,0,0)$, 
$(0,2,0,0,5,0)$, 
$(1,2,0,0,5,0)$,  
$\{(0,0,3,4,0,0)$ or $(0,0,4,3,0,0)\}$, 
$\{(1,0,3,4,0,0)$ or $(1,0,4,3,0,0)\}$, 
$\{(0,2,3,4,5,0)$ or $(0,2,4,3,5,0)\}$, 
$\{(1,2,3,4,5,0)$ or $(1,2,4,3,5,0)\}$.\\

$(ii)$ If $(\frac{\pi_1}{\pi_3})_5\,=\,1$ and $K_6\,=\,k(\sqrt[5]{\pi_1\pi_3})$ we have $[\mathcal{P}_1\mathcal{P}_3]$ capitulates in $K_6$ then if $i_6 \neq 0$ we have $i_6 = 1$. $[\mathcal{P}_1\mathcal{P}_3^4]$ capitulates in $K_1$ then if $i_1 \neq 0$ we have $i_1 = 6$, so the same possible types of capitulation accur as in $(i)$ with $i_6 = 0$ or $1$ and $i_1 = 0$ or $6$.\\

$(iii)$ If $(\frac{\pi_1}{\pi_3})_5\,\neq\,1$, by $(1)$ we have 
$K_2\,=\,k(\sqrt[5]{\pi_3})$ and $K_5\,=\,k(\sqrt[5]{\pi_1})$ then the same possible types of capitulation accur as $(i)$ and $(ii)$ by permuting 2 and 5.
\end{proof}

\subsection{The case $n\,=\,p^eq\,\equiv\,\pm1\pm7\,(\mathrm{mod}\,25)$, where $p\,\not\equiv\,1\,(\mathrm{mod}\,25)$, $q\,\not\equiv\,\pm7\,(\mathrm{mod}\,25)$}
Let $k\,=\,\Gamma(\zeta_5)$ be the normal closure of $\Gamma\,=\,\mathbb{Q}(\sqrt[5]{n})$, where $n\,=\,p^eq\,\equiv\,\pm1\pm7\,(\mathrm{mod}\,25)$ such that $p\,\not\equiv\,1,(\mathrm{mod}\,25)$, $q\,\not\equiv\,\pm7\,(\mathrm{mod}\,25)$ and $e\in \{1,2,3,4\}$. Since $q\,\equiv\,\pm2\,(\mathrm{mod}\,5)$ we have that $q$ is inert in $k_0\,=\,\mathbb{Q}(\zeta_5)$, so we can take in the following $q\,=\,\pi_5$ a prime in $k_0$. As before, by $\mathcal{P}_1, \mathcal{P}_2, \mathcal{P}_3, \mathcal{P}_4$ and $\mathcal{P}_5$ we denote respectivly the prime ideals of $k$ above $\pi_1, \pi_2, \pi_3, \pi_4$ and $\pi_5$ in $k_0$, such that $\mathcal{P}_i^5\,=\,\pi_i\mathcal{O}_k\, (i=1,2,3,4,5)$. We have that $\tau^2$ permutes $\pi_1$ with $\pi_3$, then $\tau^2$ permutes $\mathcal{P}_1$ with $\mathcal{P}_3$, but $\tau^2$ sets $q\,=\,\pi_5$ and also $\mathcal{P}_5$.\\
The six intermediate extensions of $k_5^{(1)}/k$ are determined as follows:
\begin{theorem}\label{6 exten case 2}
Let $k, n, \pi_1, \pi_2, \pi_3, \pi_4$ and $\pi_5$ as above. Put $x_1\,=\,\pi_1\pi_5^{h_1}$ and $x_2\,=\,\pi_1\pi_3^{4}$ are choosen  such that $x_1\,\equiv\,x_2\,\equiv 1\,(\mathrm{mod}\,\lambda^5)$, where $h_1 \in \{1,2,3,4\}$. Then:
\item[$(1)$] $k_5^{(1)}\,=\, k(\sqrt[5]{x_1}, \sqrt[5]{x_2})$.
\item[$(2)$] The six intermediate extensions of $k_5^{(1)}/k$ are:  $k(\sqrt[5]{x_1})$, $k(\sqrt[5]{x_2})$, $k(\sqrt[5]{\pi_1\pi_3\pi_5^{2h_1}})$, $k(\sqrt[5]{\pi_1^2\pi_3^4\pi_5^{h_1}})$, $k(\sqrt[5]{\pi_1^4\pi_3^2\pi_5^{h_1}})$ and $k(\sqrt[5]{\pi_3\pi_5^{h_1}})$. Furthermore $\tau^2$ permutes $k(\sqrt[5]{\pi_1^2\pi_3^4\pi_5^{h_1}})$ with $k(\sqrt[5]{\pi_1^4\pi_3^2\pi_5^{h_1}})$ and $k(\sqrt[5]{x_1})$ with $k(\sqrt[5]{\pi_3\pi_5^{h_1}})$, and sets $k(\sqrt[5]{x_2})$, $k(\sqrt[5]{\pi_1\pi_3\pi_5^{2h_1}})$.
\end{theorem}

\begin{proof}
Since $k\,=\,k_0(\sqrt[5]{n})$ we can write $n$ in $k_0$ as $n\,=\,\pi_1^e\pi_2^e\pi_3^e\pi_4^e\pi_5$ with $\pi_i\,\not\equiv\,1\,(\mathrm{mod}\,\lambda^5)$ because $p\,\not\equiv\,1\,(\mathrm{mod}\,25)$ and $q\,\not\equiv\,1\,(\mathrm{mod}\,25)$. By proposition \ref{prop genre} there exists $h_1, h_2 \in \{1,..,4\}$ such that $\pi_1\pi_5^{h_1}\,\equiv\, \pm1,\pm7\, (\mathrm{mod}\,\lambda^5)$ and $\pi_1\pi_3^{h_2}\,\equiv\, \pm1,\pm7\, (\mathrm{mod}\,\lambda^5)$. to investigate the correspondence between the six intermediate extension of $k_5^{(1)}/k$ and the six subgroups of $C_{k,5}$, we assume that $h_2\,=\,4$. Put $x_1\,=\,\pi_1\pi_5^{h_1}$ and $x_2\,=\,\pi_1\pi_3^{4}$.
\item[$(1)$] The fact that $k_5^{(1)}\,=\, k(\sqrt[5]{x_1}, \sqrt[5]{x_2})$ follows from proposition \ref{prop genre}.
\item[$(2)$] The six intermediate extensions are: $k(\sqrt[5]{x_1})$, $k(\sqrt[5]{x_2})$, $k(\sqrt[5]{x_1x_2})$, $k(\sqrt[5]{x_1x_2^2})$, $k(\sqrt[5]{x_1x_2^3})$ and $k(\sqrt[5]{x_1x_2^4})$. Since $x_1\,=\,\pi_1\pi_5^{h_1}$ and $x_2\,=\,\pi_1\pi_5^{4}$, we have $k(\sqrt[5]{x_1x_2})\,=\,k(\sqrt[5]{\pi_1^2\pi_3^4\pi_5^{h_1}})$, $k(\sqrt[5]{x_1x_2^2})\,=\,k(\sqrt[5]{\pi_1\pi_3\pi_5^{2h_1}})$, $k(\sqrt[5]{x_1x_2^3})\,=\,k(\sqrt[5]{\pi_1^4\pi_3^2\pi_5^{h_1}})$ and $k(\sqrt[5]{x_1x_2^4})\,=\,k(\sqrt[5]{\pi_3\pi_5^{h_1}})$. Since $\pi_1^{\tau^2}\,=\,\pi_3$, $\pi_3^{\tau^2}\,=\,\pi_1$ and $\pi_5^{\tau^2}\,=\,\pi_5$, and by the same reasoning as theorem \ref{6 exten case 1} we prove that $\tau^2$ permutes $k(\sqrt[5]{\pi_1^2\pi_3^4\pi_5^{h_1}})$ with $k(\sqrt[5]{\pi_1^4\pi_3^2\pi_5^{h_1}})$ and $k(\sqrt[5]{x_1})$ with $k(\sqrt[5]{\pi_3\pi_5^{h_1}})$, and sets $k(\sqrt[5]{x_2})$, $k(\sqrt[5]{\pi_1\pi_3\pi_5^{2h_1}})$.
\end{proof}
The generators of $C_{k,5}$ in this case are determined as follows:
\begin{theorem}\label{generator case 2} Let $k, n, \pi_1, \pi_2, \pi_3, \pi_4, \pi_5$ and $h_1$ as above. Let $\mathcal{P}_1, \mathcal{P}_2, \mathcal{P}_3, \mathcal{P}_4$ and $\mathcal{P}_5$ prime ideals of $k$ such that $\mathcal{P}_i^5\,=\,\pi_i\mathcal{O}_{k_0}\, (i=1,2,3,4,5)$. Then:
\begin{center}
$C_{k,5}\,=\,\langle [\mathcal{P}_1\mathcal{P}_3\mathcal{P}_5^{2h_1}], [\mathcal{P}_1\mathcal{P}_3^4]\rangle$
\end{center}
\end{theorem}

\begin{proof}
According to [\ref{FOU1}, theorem 1.1], for this case of the radicand $n$, we have that $\zeta_5^i(1+\zeta_5)^j$ is not norm of element in $k-\{0\}$ for any exponents $i$ and $j$, then by [\ref{Mani}, section 5.3], we have $C_{k,5}\,=\,C_{k,5}^{(\sigma)}\,=\,C_{k,s}^{(\sigma)}\,=\,\langle [\mathcal{P}_1], [\mathcal{P}_2], [\mathcal{P}_3], [\mathcal{P}_4], [\mathcal{P}_5]\rangle$. Since $\mathcal{P}_1^{\tau^2}\,=\,\mathcal{P}_3$, $\mathcal{P}_2^{\tau^2}\,=\,\mathcal{P}_4$ and $\mathcal{P}_5^{\tau^2}\,=\,\mathcal{P}_5$, as the proof of theorem \ref{generator case 1} we have that $\mathcal{P}_1, \mathcal{P}_3$ and $\mathcal{P}_5$ are non principals. As $[\mathcal{P}_1\mathcal{P}_3\mathcal{P}_5^{2h_1}]^{\tau^2}\,=\,[(\mathcal{P}_1\mathcal{P}_3\mathcal{P}_5^{2h_1})^{\tau^2}]\,=\,[\mathcal{P}_3\mathcal{P}_1\mathcal{P}_5^{2h_1}]\,=\,[\mathcal{P}_1\mathcal{P}_3\mathcal{P}_5^{2h_1}]$ then $C_{k,5}^+\,=\,\langle[\mathcal{P}_1\mathcal{P}_3\mathcal{P}_5^{2h_1}]\rangle$, and we have that $C_{k,5}^-\,=\,\langle[\mathcal{P}_1\mathcal{P}_3^4]\rangle$. Hence $C_{k,5}\,=\,\langle [\mathcal{P}_1\mathcal{P}_3\mathcal{P}_5^{2h_1}], [\mathcal{P}_1\mathcal{P}_3^4]\rangle$.
\end{proof}

The main theorem of capitulation in this case is as follows:
\begin{theorem}
We keep the same assumptions as theorem \ref{generator case 2} Then:
\item[$(1)$] $K_1\,=\,k(\sqrt[5]{\pi_1\pi_3^4})$ or $k(\sqrt[5]{\pi_1\pi_3\pi_5^{2h_1}})$, $K_2\,=\,k(\sqrt[5]{\pi_1\pi_5^{h_1}})$ or $k(\sqrt[5]{\pi_3\pi_5^{h_1}})$, $K_3\,=\,k(\sqrt[5]{\pi_1^2\pi_3^4\pi_5^{h_1}})$ or $k(\sqrt[5]{\pi_1^4\pi_3^2\pi_5^{h_1}})$ , $K_4\,=\,k(\sqrt[5]{\pi_1^4\pi_3^2\pi_5^{h_1}})$ or $k(\sqrt[5]{\pi_1^2\pi_3^4\pi_5^{h_1}})$, $K_5\,=\,k(\sqrt[5]{\pi_3\pi_5^{h_1}})$ or $k(\sqrt[5]{\pi_1\pi_5^{h_1}})$  and $K_6\,=\,k(\sqrt[5]{\pi_1\pi_3\pi_5^{2h_1}})$ or $k(\sqrt[5]{\pi_1\pi_3^4})$.

\item[$(2)$] $[\mathcal{P}_1\mathcal{P}_3\mathcal{P}_5^{2h_1}]$ capitulates in $k(\sqrt[5]{\pi_1\pi_3\pi_5^{2h_1}})$, $[\mathcal{P}_1\mathcal{P}_5^{h_1}]$ capitulates in $k(\sqrt[5]{\pi_1\pi_5^{h_1}})$, $[\mathcal{P}_1^2\mathcal{P}_3^4\mathcal{P}_5^{h_1}]$ capitulates in $k(\sqrt[5]{\pi_1^2\pi_3^4\pi_5^{h_1}})$, $[\mathcal{P}_1^4\mathcal{P}_3^2\mathcal{P}_5^{h_1}]$ capitulates in $k(\sqrt[5]{\pi_1^4\pi_3^2\pi_5^{h_1}})$, $[\mathcal{P}_3\mathcal{P}_5^{h_1}]$ capitulates in $k(\sqrt[5]{\pi_3\pi_5^{h_1}})$ and $[\mathcal{P}_1\mathcal{P}_3^{4}]$ capitulates in $k(\sqrt[5]{\pi_1\pi_3^{4}})$.

\item[$(3)$] - If $K_1\,=\, k(\sqrt[5]{\pi_1\pi_3\pi_5^{2h_1}})$, then the possible types of capitulation are: $(0,0,0,0,0,0)$, $(1,0,0,0,0,0)$, 
$\{(0,5,0,0,2,0)$ or $(0,2,0,0,5,0)\}$, 
$\{(1,5,0,0,2,0)$ or $(1,2,0,0,5,0)\}$, 
$\{(0,5,4,3,2,0)$ or $(0,2,4,3,5,0)\}$, 
$\{(1,5,4,3,2,0)$ or $(1,2,4,3,5,0)\}$,
$\{(0,5,3,4,2,0)$ or $(0,2,3,4,5,0)\}$,
$\{(1,5,3,4,2,0)$ or $(1,2,3,4,5,0)\}$,
$\{(0,0,3,4,0,0)$ or $(0,0,4,3,0,0)\}$,
$\{(1,0,3,4,0,0)$ or $(1,0,4,3,0,0)\}$.\\
- If $K_1\,=\,k(\sqrt[5]{\pi_1\pi_3^4})$, then the same possible types occur, with $i_6$ takes value 0 or 1.
\end{theorem}
\begin{proof}
\item[$(1)$] According to theorem \ref{6 exten case 2}, we have that $\tau^2$ permutes $k(\sqrt[5]{\pi_1^2\pi_3^4\pi_5^{h_1}})$ with $k(\sqrt[5]{\pi_1^4\pi_3^2\pi_5^{h_1}})$ and $k(\sqrt[5]{x_1})$ with $k(\sqrt[5]{\pi_3\pi_5^{h_1}})$, and sets $k(\sqrt[5]{x_2})$, $k(\sqrt[5]{\pi_1\pi_3\pi_5^{2h_1}})$. We determine first the six subgroups $H_i$ of $C_{k,5}$. We have that $C_{k,5}\,=\,\langle\mathcal{A}, \mathcal{X}\rangle$, where $H_1\,=\,C_{k,5}^+\,=\,\langle \mathcal{A} \rangle$ and $H_6\,=\,C_{k,5}^-\,=\,\langle \mathcal{X} \rangle$. By theorem \ref{generator case 2} we have $\mathcal{A}\,=\,[\mathcal{P}_1\mathcal{P}_3\mathcal{P}_5^{2h_1}]$ and $\mathcal{X}\,=\,[\mathcal{P}_1\mathcal{P}_3^4]$, then $\mathcal{A}\mathcal{X}\,=\,[\mathcal{P}_1\mathcal{P}_5^{h_1}]^2$, $\mathcal{A}\mathcal{X}^2\,=\,[\mathcal{P}_1^2\mathcal{P}_3^4\mathcal{P}_5^{h_1}]^4$, $\mathcal{A}\mathcal{X}^3\,=\,[\mathcal{P}_1^4\mathcal{P}_3^2\mathcal{P}_5^{h_1}]$ and $\mathcal{A}\mathcal{X}^4\,=\,[\mathcal{P}_3\mathcal{P}_5^{h_1}]^3$. Hence $H_2\,=\,\langle [\mathcal{P}_1\mathcal{P}_5^{h_1}]\rangle$, $H_3\,=\,\langle [\mathcal{P}_1^2\mathcal{P}_3^4\mathcal{P}_5^{h_1}] \rangle$, $H_4\,=\,\langle [\mathcal{P}_1^4\mathcal{P}_3^2\mathcal{P}_5^{h_1}]\rangle$ and $H_5\,=\,\langle [\mathcal{P}_3\mathcal{P}_5^{h_1}] \rangle$. Since $\tau^2$ sets $k(\sqrt[5]{\pi_1\pi_3\pi_5^{2h_1}})$ and $k(\sqrt[5]{\pi_1\pi_3^4})$, so if $K_1\,=\,k(\sqrt[5]{\pi_1\pi_3\pi_5^{2h_1}})$ then $K_6\,=\,k(\sqrt[5]{\pi_1\pi_3^4})$ and inversly. By class field theory, the fact that $H_i\,(i = 2,5)$ correspond to $K_i\,(i = 2,5)$ mean that $\mathcal{P}_1\mathcal{P}_5^{h_1}$ splits completly in $K_2$ and $\mathcal{P}_3\mathcal{P}_5^{h_1}$ splits completly in $K_5$. As $\pi_1\pi_5^{h_1}$ divides $\pi_1^2\pi_3^4\pi_5^{h_1}$ and $\pi_1^4\pi_3^2\pi_5^{h_1}$, by proposition \ref{prKummer}, $\pi_1\pi_5^{h_1}$ can not split in $k_0{\sqrt[5]{\pi_1^2\pi_3^4\pi_5^{h_1}}}$ and $k_0{\sqrt[5]{\pi_1^4\pi_3^2\pi_5^{h_1}}}$, this equivalent to say that $\mathcal{P}_1\mathcal{P}_5^{h_1}$ can not split completly in $k{\sqrt[5]{\pi_1^2\pi_3^4\pi_5^{h_1}}}$ and $k{\sqrt[5]{\pi_1^4\pi_3^2\pi_5^{h_1}}}$. By the same reasoning we have that  
$\mathcal{P}_3\mathcal{P}_5^{h_1}$ can not split completly in $k{\sqrt[5]{\pi_1^2\pi_3^4\pi_5^{h_1}}}$ and $k{\sqrt[5]{\pi_1^4\pi_3^2\pi_5^{h_1}}}$. Hence if $K_2\,=\,k(\sqrt[5]{\pi_1\pi_5^{h_1}})$ then $K_5\,=\,k(\sqrt[5]{\pi_3\pi_5^{h_1}})$ and inversly, which allow us to deduce that if $K_3\,=\,k(\sqrt[5]{\pi_1^2\pi_3^4\pi_5^{h_1}})$ then $K_5\,=\,k(\sqrt[5]{\pi_1^4\pi_3^2\pi_5^{h_1}})$ and inversly.

\item[$(2)$]We keep the same proof as $(2)$ theorem \ref{theo capi case 1}.

\item[$(3)$]-If $K_1\,=\, k(\sqrt[5]{\pi_1\pi_3\pi_5^{2h_1}})$, then $K_6\,=\,k\Gamma_5^{(1)}\,=\,k(\sqrt[5]{\pi_1\pi_3^4})$ and we have that $[\mathcal{P}_1\mathcal{P}_3^{4}]$ capitulates in $K_6$, moreover since $C_{k,5}^+\,=\,\langle[\mathcal{P}_1\mathcal{P}_3\mathcal{P}_5^{2h_1}]\rangle \cong C_{\Gamma,5}$ then $\mathcal{P}_1\mathcal{P}_3\mathcal{P}_5^{2h_1}\,=\,j_{k/\Gamma}(\mathcal{J})$ such that $C_{\Gamma,5}\,=\,\langle\mathcal{J}\rangle$, then $[\mathcal{P}_1\mathcal{P}_3\mathcal{P}_5^{2h_1}]$ capitulates in $K_6$. As $C_{k,5}\,=\,\langle [\mathcal{P}_1\mathcal{P}_3\mathcal{P}_5^{2h_1}], [\mathcal{P}_1\mathcal{P}_3^4]\rangle$, then all classes capitulate in $K_6\,=\,k(\sqrt[5]{\pi_1\pi_3^4})$. We determine the possible types of capitulation $(i_1,i_2,i_3,i_4,i_5,i_6)$. We have that $i_6\,=\,0$, $K_2\,=\,K_5^{\tau^2}$, $K_3\,=\,K_4^{\tau^2}$ and $C_{k,5}\,=\,C_{k,5}^{\tau^2}$. If $i_1\neq 0$ we have $i_1 = 1$. $i_2$ and $i_5$ are both nulls or non nulls, so if $i_2$ and $i_5$ $\neq 0$, then $(i_2, i_5)\,=\,(2,5)$ or $(5,2)$ depending on $\mathcal{P}_1\mathcal{P}_5^{h_1}$ splits completely in $k(\sqrt[5]{\pi_1\pi_5^{h_1}})$ or in $k(\sqrt[5]{\pi_3\pi_5^{h_1}})$. Similarly if $i_3$ and $i_4$ $\neq 0$, then $(i_3, i_4)\,=\,(3,4)$ or $(4,3)$. Hence the possible types given are proved.\\
-If $K_1\,=\, k(\sqrt[5]{\pi_1\pi_3^4})$ then $K_6\,=\,k\Gamma_5^{(1)}\,=\,k(\sqrt[5]{\pi_1\pi_3\pi_5^{2h_1}})$ and we have 
$C_{k,5}^+\,=\,\langle[\mathcal{P}_1\mathcal{P}_3\mathcal{P}_5^{2h_1}]\rangle$ capitulates in $K_6$, the possible values of $i_2, i_3, i_4, i_5$ are as above, $(i_2, i_5)\,=\,(2,5)$ or $(5,2)$ if they are non nulls, $(i_3, i_4)\,=\,(3,4)$ or $(4,3)$ if they are non nulls. If $i_1 \neq 0$ then $i_1 = 6$ because $H_6\,=\,\langle[\mathcal{P}_1\mathcal{P}_3^4]\rangle$, and if $i_6 \neq 0$ then $i_1 = 1$ because $H_1\,=\,\langle[\mathcal{P}_1\mathcal{P}_3\mathcal{P}_5^{2h_1}]\rangle$. Hence the possible types given are proved.
\end{proof}

\subsection{The case $n\,=\,5^ep\,\not\equiv\,\pm1\pm7\,(\mathrm{mod}\,25)$, where $p\,\not\equiv\,1\,(\mathrm{mod}\,25)$}
Let $k\,=\,\Gamma(\zeta_5)$ be the normal closure of $\Gamma\,=\,\mathbb{Q}(\sqrt[5]{n})$, where $n\,=\,5^ep$ such that $p\,\not\equiv\,1,(\mathrm{mod}\,25)$ and $e\in \{1,2,3,4\}$. Since $n\,=\,5^ep\,\not\equiv\,\pm1\pm7,(\mathrm{mod}\,25)$ then $\lambda\,=\,1-\zeta_5$ is ramified in $k/k_0$. Let $\pi_1, \pi_2, \pi_3$ and $\pi_4$ primes of $k_0$ such that $p\,=\,\pi_1\pi_2\pi_3\pi_4$. Let $\mathcal{P}_1, \mathcal{P}_2, \mathcal{P}_3, \mathcal{P}_4$ and $\mathcal{I}$ prime ideals of $k$ above $\pi_1, \pi_2, \pi_3, \pi_4$
and $\lambda$, we have $\mathcal{P}_i^5\,=\,\pi_i\mathcal{O}_k$ and 
$\mathcal{I}^5\,=\,\lambda\mathcal{O}_k$. According to [\ref{FOU1}, theorem 1.1], for this case of the radicand $n$, we have that $\zeta_5^i(1+\zeta_5)^j$ is not norm of element in $k-\{0\}$ for any exponents $i$ and $j$, then we have $C_{k,5}\,=\,C_{k,5}^{(\sigma)}\,=\,C_{k,s}^{(\sigma)}$. Hence the results  about the six intermediate extensions of $k_5^{(1)}/k$, the generators of $C_{k,5}$ and the capitulation problem in this case are the same as case $2$ by substituting $q$ by $5$, $\pi_5$ by $\lambda\,=\,1-\zeta_5$ and $\mathcal{P}_5$ by $\mathcal{I}$. 

\begin{center}
\section{Numerical examples}
\end{center}
The task to determine the capitulation in a cyclic quintic extension of a base field of degree $20$, that is, in a field of absolute degree $100$, is definitely far beyond the reach of computational algebra systems like MAGMA and Pari/GP. For this reason we give exemples of a pure metacyclic fields $k\,=\,\mathbb{Q}(\sqrt[5]{n}, \zeta_5)$ such that $C_{k,5}$ is of type $(5,5)$ and $C_{k,5}\,=\,C_{k,5}^{(\sigma)}$.\\

\begin{center}
Table 1: $k\,=\,\mathbb{Q}(\sqrt[5]{n}, \zeta_5)$ with $C_{k,5}$ is of type $(5,5)$ and $C_{k,5}\,=\,C_{k,5}^{(\sigma)}$.\\

\begin{tabular}{|c c c c|c|c c c c c|}
\hline 
  $n$ & $h_{k,5}$ & $C_{k,5}$ & rank $(C_{k,5}^{(\sigma)})$&  & $n$ & $h_{k,5}$ & $C_{k,5}$ & rank $(C_{k,5}^{(\sigma)})$ &\\ 
\hline 
$55$ &  25 & $(5,5)$ & 2& &$1457$ &  25 & $(5,5)$ & 2 & \\ 
$655$ &  25 & $(5,5)$ & 2& &$6943$ &  25 & $(5,5)$ & 2 &\\ 
$1775$ &  25 & $(5,5)$ & 2& &$8507$ &  25 & $(5,5)$ & 2 & \\ 
$1555$ &  25 & $(5,5)$ & 2& &$12707$ &  25 & $(5,5)$ & 2 &\\
$2155$ &  25 & $(5,5)$ & 2& &$151$ &  25 & $(5,5)$ & 2 & \\ 
$5125$ &  25 & $(5,5)$ & 2& &$1301$ &  25 & $(5,5)$ & 2 &\\
$8275$ &  25 & $(5,5)$ & 2& &$2111$ &  25 & $(5,5)$ & 2 & \\ 
$30125$ &  25 & $(5,5)$ & 2& &$251^2$ &  25 & $(5,5)$ & 2 &\\ 
$38125$ &  25 & $(5,5)$ & 2& &$601^3$ &  25 & $(5,5)$ & 2 & \\ 
$113125$ &  25 & $(5,5)$ & 2& &$2131^2$ &  25 & $(5,5)$ & 2 &\\
$93$ &  25 & $(5,5)$ & 2& &$1901^4$ &  25 & $(5,5)$ & 2 & \\ 
$382$ &  25 & $(5,5)$ & 2& &$1051^4$ &  25 & $(5,5)$ & 2 &\\
$943$ &  25 & $(5,5)$ & 2& &$1801^3$ &  25 & $(5,5)$ & 2 &\\

\hline 
\end{tabular} 
\end{center}

\newpage


\end{document}